\documentclass[11pt]{article}
\usepackage[utf8]{inputenc}
\usepackage[T1]{fontenc}
\usepackage[english]{babel}
\usepackage[margin=3cm]{geometry}

\usepackage{listings}
\usepackage{amsmath}
\usepackage{amsfonts}
\usepackage{amssymb}
\usepackage{parskip}
\usepackage{amsthm}
\usepackage{mathtools}
\usepackage{thmtools}
\usepackage{etoolbox}
\usepackage{cite}
\usepackage{color}
\usepackage{subcaption}
\usepackage{todonotes}
\usepackage{longtable}
\usepackage{booktabs}
\usepackage{siunitx}

\begingroup
    \makeatletter
    \@for\theoremstyle:=definition,remark,plain\do{%
        \expandafter\g@addto@macro\csname th@\theoremstyle\endcsname{%
            \addtolength\thm@preskip\parskip
            }%
        }
\endgroup

\usepackage{cite}

\usepackage{microtype}	
\usepackage[colorlinks,citecolor=red]{hyperref}
\usepackage[nameinlink]{cleveref}
\newtheorem{theorem}{Theorem}[section]
\crefname{theorem}{theorem}{theorems}
\Crefname{theorem}{Theorem}{Theorems}

\newtheorem{assumption}[theorem]{Assumption}
\crefname{assumption}{assumption}{assumptions}
\Crefname{assumption}{Assumption}{Assumptions}

\newtheorem{prop}[theorem]{Proposition}
\crefname{prop}{proposition}{propositions}
\Crefname{prop}{Proposition}{Propositions}

\newtheorem{lemma}[theorem]{Lemma}
\crefname{lemma}{lemma}{lemmas}
\Crefname{lemma}{Lemma}{Lemmas}

\newtheorem{corollary}[theorem]{Corollary}
\crefname{corollary}{corollary}{corollaries}
\Crefname{corollary}{Corollary}{Corollaries}

\theoremstyle{definition}
\newtheorem{remark}[theorem]{Remark}
\AtEndEnvironment{remark}{\hfill\ensuremath{\diamondsuit}}
\crefname{remark}{remark}{remarks}
\Crefname{remark}{Remark}{Remarks}

\theoremstyle{definition}
\newtheorem{defn}[theorem]{Definition}
\crefname{defn}{definition}{definitions}
\Crefname{defn}{Definition}{Definitions}

\newtheorem{example}[theorem]{Example}
\AtEndEnvironment{example}{\hfill\ensuremath{\bigcirc}}
\crefname{example}{example}{example}
\Crefname{example}{Example}{Examples}

\newcommand{\field}{\mathbb{F}}
\newcommand{\GL}[2]{\mathrm{GL}_{#1}(#2)}
\newcommand{\E}{E}
\newcommand{\A}{A}
\newcommand{\B}{D}
\newcommand{\D}{B}
\newcommand{\f}{f}
\newcommand{\state}{x}
\newcommand{\z}{z}
\newcommand{\derv}[2]{#1^{(#2)}}
\newcommand{\stateDim}{n}
\newcommand{\delay}{\tau}
\newcommand{\shift}[1]{\sigma_{#1}}
\newcommand{\history}{\phi}

\newcommand{\timeInt}{\mathbb{I}}
\newcommand{\tf}{t_\mathrm{f}}
\newcommand{\timeIntEx}{[0,\tf]}
\newcommand{\maxMOS}{M}
\newcommand{\mosIndexSet}{\mathcal{I}}
\newcommand{\mosIndexSetEx}{\{1,\ldots,\maxMOS\}}
\newcommand{\mos}[2]{#1_{#2}}
\newcommand{\dmos}[2]{\dot{#1}_{#2}}
\newcommand{\mosf}[1]{\mos{\f}{#1}}
\newcommand{\mosfDAE}[1]{\mos{\tilde{\f}}{#1}}

\newcommand{\WCFL}{S}
\newcommand{\WCFR}{T}
\newcommand{\ndif}{\stateDim_{\mathrm{d}}}
\newcommand{\nalg}{\stateDim_{\mathrm{a}}}
\newcommand{\J}{J}
\newcommand{\N}{N}
\newcommand{\ind}{\mathrm{ind}}
\newcommand{\indEA}{{\nu}}
\newcommand{\sstate}{v}
\newcommand{\ssf}{g}
\newcommand{\smosf}[1]{\mos{\ssf}{#1}}
\newcommand{\smosfDAE}[1]{\mos{\tilde{\ssf}}{#1}}
\newcommand{\fstate}{w}
\newcommand{\ff}{h}
\newcommand{\fmosf}[1]{\mos{\ff}{#1}}
\newcommand{\fmosfDAE}[1]{\mos{\tilde{\ff}}{#1}}
\newcommand{\shistory}{\psi}
\newcommand{\fhistory}{\eta}
\newcommand{\Bd}{\B_\mathrm{d}}
\newcommand{\Ba}{\B_\mathrm{a}}
\newcommand{\Bdd}{\B_{\mathrm{d},1}}
\newcommand{\Bda}{\B_{\mathrm{d},2}}
\newcommand{\Bad}{\B_{\mathrm{a},1}}
\newcommand{\Baa}{\B_{\mathrm{a},2}}

\newcommand{\Adiff}{\A^{\mathrm{diff}}}
\newcommand{\Acon}{\A^{\mathrm{con}}}
\newcommand{\C}[1]{C_{#1}}
\newcommand{\Bk}[1]{\B_{#1}}
\newcommand{\Dk}[1]{\D_{#1}}

\newcommand{\retarded}{smoothing type}
\newcommand{\neutral}{discontinuity invariant type}
\newcommand{\advanced}{de-smoothing type}

\newcommand{\indicator}[1]{\Delta_{#1}}

\newcommand{\bE}{\mathcal{E}}
\newcommand{\bA}{\mathcal{A}}
\newcommand{\bB}{\mathcal{B}}
\newcommand{\bstate}{\zeta}
\newcommand{\bff}{\mathcal{F}}

\title{Discontinuity propagation in delay differential-algebraic equations\footnotemark[1]}
\author{Benjamin Unger\footnotemark[2]}
\date{\today}

\begin{document}

\maketitle
\renewcommand{\thefootnote}{\fnsymbol{footnote}}
\footnotetext[1]{Research supported by the DFG Collaborative Research Center 910 \emph{Control of self-organizing nonlinear systems: Theoretical methods and concepts of application}, project A2.}
\footnotetext[2]{
Institut f\"ur Mathematik,
Technische Universität Berlin, Str.\ des 17.~Juni~136,
10623~Berlin,
Federal Republic of Germany,
\texttt{unger@math.tu-berlin.de}.}
%
\begin{abstract}
 	The propagation of primary discontinuities in initial value problems for linear delay differential-algebraic equations (DDAEs) is discussed. Based on the (quasi-) Weierstra{\ss} form for regular matrix pencils, a complete characterization of the different propagation types is given and algebraic criteria in terms of the matrices are developed. The analysis, which is based on the method of steps, takes into account all possible inhomogeneities and history functions and thus serves as a worst-case scenario. Moreover, it reveals possible hidden delays in the DDAE and allows to study exponential stability of the DDAE based on the spectral abscissa. The new classification for DDAEs is compared to existing approaches in the literature and the impact of splicing conditions on the classification is studied.
\vskip .3truecm

\vskip .3truecm

{\bf AMS(MOS) subject classification:} 34A09, 34A12, 34K06, 65H10
\end{abstract}


\section{Introduction}
In this paper we study \emph{delay differential-algebraic equations} (DDAEs) of the form
\begin{subequations}
	\label{eq:IVP}
	\begin{equation}
		\label{eq:DDAE}
		\E\dot{\state}(t) = \A\state(t) + \B\state(t-\delay) + \f(t)
	\end{equation}
	in the time interval $\timeInt\vcentcolon=\timeIntEx$, where $\E,\A,\B\in\field^{\stateDim,\stateDim}$ are matrices, $f:\timeInt\to\field^{\stateDim}$ is the inhomogeneity, $\dot{\state}$ denotes the time derivative $\mathrm{d}/\mathrm{d}t$ of $x$ from the right, and the field $\field$ is either the complex or the real numbers, i.\,e., $\field\in\{\mathbb{R},\mathbb{C}\}$. Often, \eqref{eq:DDAE} is formulated as an initial value problem (IVP), i.\,e., we equip \eqref{eq:DDAE} with the initial condition 
	\begin{equation}
		\label{eq:initialCondition}
		\state(t) = \history(t)\qquad\text{for } t\in[-\delay,0]
	\end{equation}
with \emph{history function} $\history:[-\tau,0]\to\field^{\stateDim}$.
\end{subequations}
DDAEs of the form \eqref{eq:DDAE} arise as linearization of the nonlinear implicit equation 
\begin{displaymath}
	F(t,\state(t),\dot{\state}(t),\state(t-\delay)) = 0
\end{displaymath}
around a nominal stationary solution. Typical applications are nonlinear optics, chemical reactor systems and delayed feedback control (see \cite{Ern09} and the references therein). Moreover, the DDAE \eqref{eq:DDAE} may result from a realization of a transport-dominated phenomenon \cite{SchU16,SchUBG18}.

It is well-known, that the history function $\history$ may not be linked smoothly to the solution $\state$ at $t = 0$. More precisely, we have
\begin{equation}
	\label{eq:discontinuity}
	\lim_{t\nearrow0}\dot{\history}(t)\neq \lim_{t\searrow0} \dot{\state}(t)
\end{equation}
in general. Due to the delay, this so called \emph{primary discontinuity} \cite{BelZ03} is propagated to integer multiples of the delay $\delay$. Thus a rigorous analysis of the regularity of the solution is important for any kind of numerical integrator that is based on a Taylor series expansion of the solution. If $\E = I_{\stateDim}$ is the $\stateDim\times\stateDim$ identity matrix, the DDAE \eqref{eq:DDAE} is called \emph{retarded delay differential equation} (rDDE) and it is well-known that in this case, the primary discontinuities are smoothed out, i.\,e. if 
\begin{displaymath}
	\lim_{t\nearrow k\delay}\state^{(j-1)}(t) = \lim_{t\searrow k\delay} \state^{(j-1)}(t)\qquad\text{and}\qquad \lim_{t\nearrow k\delay}\state^{(j)}(t)\neq \lim_{t\searrow k\delay} \state^{(j)}(t)
\end{displaymath}
holds for some $j,k\in\mathbb{N}$, then we have
\begin{displaymath}
	\lim_{t\nearrow (k+1)\delay}\state^{(j)}(t)= \lim_{t\searrow (k+1)\delay} \state^{(j)}(t),
\end{displaymath}
provided that $\f$ is smooth enough. This situation is specific to the case that the matrix $\E$ is nonsingular. If in contrast, the matrix $\E$ is singular, the situation is completely different (see also \cite{Cam91}), as the following two examples suggest.
\begin{example}
	\label{ex:neutral}
	Let $\field = \mathbb{R}$, $\stateDim = 1$, $\E=0$, $\A = 1$, $\B = 1$, $\f\equiv 1$, $\delay=1$, and $\history(t) = t$. Then 
	\begin{displaymath}
		\state(t) = \begin{cases}
			k-1-t, & \text{if}\ k-1\leq t\leq k\ \text{and}\ k\in\mathbb{N}\ \text{odd},\\
			t+k, & \text{if}\ k-1\leq t\leq k\ \text{and}\ k\in\mathbb{N}\ \text{even}
		\end{cases}
	\end{displaymath}
	solves the initial values problem  \eqref{eq:IVP}. In particular, $\state$ is continuous for $t>-1$ but $\dot{\state}$ is discontinuous at every $t = k$ and thus no smoothing occurs.
\end{example}
\begin{example}
	\label{ex:advanced}
	Let $\field = \mathbb{R}$, $\stateDim =2$, $\f\equiv 0$, $\delay=1$, and
	\begin{displaymath}
		\E = \begin{bmatrix}
			1 & 0\\
			0 & 0
		\end{bmatrix},\qquad \A = \begin{bmatrix}
			0 & 1\\
			1 & 0
		\end{bmatrix},\qquad \B = \begin{bmatrix}
			0 & 0\\
			0 & -1
		\end{bmatrix},\qquad
		\history(t) = \begin{bmatrix}
			\frac{1}{3}(t-1)^3 + (t-1)^2 -1,\\
			\frac{1}{3}t^3 + t^2 -1
		\end{bmatrix}.
	\end{displaymath}
	Denoting the second component of $\state$ with $\state_2$, the DDAE \eqref{eq:DDAE} implies
	\begin{displaymath}
		\state_2(t) = \begin{cases}
			t^2 -1, & t\in[0,1],\\
			2t -2, & t\in[1,2],\\
			2, & t\in[2,3),\\
			0, & t\geq 3.
		\end{cases}
	\end{displaymath}
	In particular, the solution becomes less smooth at multiples of the time delay and even discontinuous at $t=3$.
\end{example}
Following the classification for scalar delay differential equations (DDEs) proposed in \cite{BelC63}, we observe, that the DDAE in \Cref{ex:neutral} is of \emph{neutral type}, while the second component in \Cref{ex:advanced} satisfies a DDE of \emph{advanced type}. The reason for this behavior is the so called \emph{index} of the differential-algebraic equation (DAE) that is encoded with the matrix pencil $(\E,\A)$. The index is, roughly speaking, a measure for the smoothness requirements for the inhomogeneity $\f$ for a solution to exist. For a detailed analysis of the different index concepts, we refer to \cite{Meh15,KunM06}.

The different classification approaches for DDAEs present in the literature, are either restricted to DDAEs in Hessenberg from with index less or equal three \cite{AscP95}, or are based on the so called \emph{underlying DDE} \cite{HaM15}. In particular, neither of the approaches reflects the propagation of primary discontinuities and the effect of so called \emph{splicing conditions} \cite{BelZ03} on the regularity of the solution. The main contributions of this work are the following:
\begin{enumerate}
	\item We introduce a new classification for DDAE based on the propagation of primary discontinuities (\Cref{def:classification}) and give a complete characterization of the propagation of discontinuities in terms of the matrices $\E,\A$, and $\B$ in \eqref{eq:DDAE}, cf. \Cref{thm:classificationDDAE}.
	\item In \Cref{cor:hiddenDelays} we show that multiple delays might be hidden in \eqref{eq:DDAE} and provide a reformulation that is suitable for the stability analysis. Moreover, we show that the new classification provides a sufficient condition to analyze the stability of the DDAE in terms of the spectral abscissa (\Cref{cor:stability}).
	\item \Cref{ex:weekleDeSmoothing} illustrates that splicing conditions can have an impact on the solvability of DDAEs. Moreover, we characterize sufficient conditions for DDAEs up to index $3$ to have a unique solution (cf. \Cref{thm:index3}).
	\item We show (\Cref{cor:classification}) that in some sense the classification introduced \cite{HaM15} is an upper bound for the classification introduced in this paper.
\end{enumerate}

\section*{Nomenclature}
\begin{center}
\begin{longtable}{p{42pt}p{12pt}p{320pt}}
	$\mathbb{N}$ && the set of natural numbers\\
	$\mathbb{N}_0$ & $\vcentcolon=$ & $\{0\}\cup\mathbb{N}$\\
	$I_{\stateDim}$ && identity matrix of size $\stateDim\times \stateDim$\\
	$\mosIndexSet$ &$\vcentcolon=$& $\mosIndexSetEx$\\
	$\timeInt$ & $\vcentcolon=$ & $\timeIntEx \subseteq [0,\maxMOS\delay]$\\
	$\field$ && either the field of real numbers $\mathbb{R}$ or the field of complex numbers $\mathbb{C}$\\
	$\field^{\stateDim,m}$ && matrices of size $\stateDim\times m$ over the field $\field$\\
	$\GL{n}{\field}$ & $\vcentcolon=$ & $\{A\in \field^{\stateDim,\stateDim}\mid A$ nonsingular $\}$\\
	$\shift{\delay}$ && shift (backward) operator: $(\shift{\delay}\state)(t) \vcentcolon= \state(t-\delay)$\\
	$\dot{\state}$ &$\vcentcolon=$& $\frac{\mathrm{d}}{\mathrm{d}t}\state$, the derivative of $\state$ from the right\\
	$\derv{\state}{j}$ &$\vcentcolon=$& $\left(\frac{\mathrm{d}}{\mathrm{d}t}\right)^jx$\\[0.2em]
	$\mos{\state}{i}(t)$ &$\vcentcolon=$ & $\state(t+(i-1)\delay)$ for $t\in[0,\delay]$\\
	$\mathcal{C}^k(\timeInt,\field^{\stateDim})$ && the vector space of all $k$-times continuously differentiable functions from the real interval $\timeInt$ into $\field^{\stateDim}$\\
	$\derv{\state}{j}(t^-)$ & $\vcentcolon=$ & $\lim_{s\nearrow t} \derv{\state}{j}(s)$\\
	$\derv{\state}{j}(t^+)$ & $\vcentcolon=$ & $\lim_{s\searrow t} \derv{\state}{j}(s)$\\
	$\indicator{(t_1,t_2]}$ & $\vcentcolon=$ & $\begin{cases}1,&\text{if } t\in(t_1,t_2],\\0,&\text{otherwise}.\end{cases}$ 
\end{longtable}
\end{center}

\section{Preliminaries and DAE theory}
\label{sec:prelim}
In this section, we review basic facts about DAE theory for linear time invariant systems. For convenience, we omit the time argument whenever possible and use the shift (backward) operator $\shift{\delay}$ defined via
\begin{displaymath}
	(\shift{\delay} \state)(t) = \state(t-\delay)
\end{displaymath}
instead, such that such that \eqref{eq:DDAE} is given by
\begin{equation}
	\label{eq:DDAE2}
	\E\dot{\state} = \A\state + \B\shift{\delay}\state + \f.
\end{equation}
Note that the formulation of the DDAE \eqref{eq:DDAE2} is not restricted to one single delay, since multiple commensurate delays \cite{Glu02} may be rewritten as a single delay by introducing new variables \cite{Ha15}. A standard approach to solve the initial value problem \eqref{eq:IVP} is via successive integration of \eqref{eq:DDAE2} on the time intervals $[(i-1)\delay,i\delay]$, which is sometimes referred to as \emph{method of steps} \cite{HaM15}. More precisely, assume that $\maxMOS$ is the smallest integer such that $\tf < \maxMOS\delay$ and introduce for $i\in\mosIndexSet \vcentcolon= \mosIndexSetEx$ the functions
\begin{equation}
	\label{eq:restriction}
	\begin{aligned}
	\mos{\state}{i}&:[0,\delay]\to\field^{\stateDim},\qquad & t &\mapsto (\shift{(1-i)\delay}\state)(t) =  \state(t+(i-1)\delay),\\
	\mos{\f}{i}&:[0,\delay]\to\field^{\stateDim}, & t &\mapsto (\shift{(1-i)\delay}\f)(t) =  \f(t+(i-1)\delay),\\
	\mos{\state}{0}&:[0,\delay]\to\field^{\stateDim}, & t &\mapsto \history(t-\delay).
	\end{aligned}
\end{equation}
Then we have to solve for each $i\in\mosIndexSetEx$ the DAE 
\begin{subequations}
	\label{eq:sequenceDAEs}
	\begin{align}
		\label{eq:methodOfSteps}	
		\E\dmos{\state}{i}  &= \A\mos{\state}{i} + \mosfDAE{i}, & t\in[0,\delay],\\
		\label{eq:initialConditionDAE}\mos{\state}{i}(0) &= \mos{\state}{i-1}(\delay).
	\end{align}
\end{subequations}
with $\mosfDAE{i} \vcentcolon= \B\mos{\state}{i-1}+\mos{\f}{i}$. For the analysis of \eqref{eq:sequenceDAEs} we employ the following solution concept from \cite{KunM06}. A function $\mos{\state}{i}\in\mathcal{C}^1([0,\delay],\field^{\stateDim})$ is called a \emph{solution} of \eqref{eq:methodOfSteps}, if it satisfies \eqref{eq:methodOfSteps} pointwise. The function $\mos{\state}{i}\in\mathcal{C}^1([0,\delay],\field^{\stateDim})$ is called a \emph{solution of the initial value problem} \eqref{eq:sequenceDAEs} if it is a solution of \eqref{eq:methodOfSteps} and satisfies \eqref{eq:initialConditionDAE}. An initial condition $\mos{\state}{i-1}(\delay)$ is called \emph{consistent}, if the initial value problem \eqref{eq:sequenceDAEs} hast at least one solution. 

The solvability of \eqref{eq:methodOfSteps} is closely connected to the matrix pencil $(\E,\A)$ and the smoothness of the inhomogeneity $\mosfDAE{i}$. If the inhomogeneity or some of its derivatives are discontinuous at certain points, we call this a \emph{secondary discontinuity} \cite{BelZ03}. For a numerical integrator, the secondary discontinuities need to be included in the time grid. However, to simplify our discussion, we assume that $\mosfDAE{i}$ is arbitrarily smooth on $(0,\delay)$. A sufficient assumption to guarantee this is to assume the following.

\begin{assumption}
	\label{ass:smoothness} The history function $\history:[-\delay,0]\to\field^{\stateDim}$ and the inhomogeneity $\f:\timeInt\to\field^{\stateDim}$ are infinitely many times continuously differentiable.
\end{assumption}

Another important assumption that we make throughout the text is the following.

\begin{assumption}
	\label{ass:regular}
	The matrix pencil $(\E,\A)$ is \emph{regular}, i.\,e., there exists $\lambda\in\field$ such that $\det\left(\lambda \E-\A\right)\neq0$.
\end{assumption}

Invoking \Cref{ass:regular,ass:smoothness}, the IVP \eqref{eq:sequenceDAEs} has a classical solution (cf. \cite{KunM06} and the discussion below) if the initial condition $\mos{\state}{i-1}(\delay)$ satisfies some algebraic equation. Hereby, $\mos{\state}{i}$ is called a \emph{(classical) solution}, if $\mos{\state}{i}$ is continuously differentiable and satisfies \eqref{eq:methodOfSteps} pointwise. If \eqref{eq:sequenceDAEs} has a unique solution $\mos{\state}{i}$ for every $i\in\mosIndexSet$, we can construct the solution $\state$ of the IVP \eqref{eq:IVP} by setting $\state(t) = \mos{\state}{i}(t - (i-1)\delay)$ for $t\in [(i-1)\delay,i\delay]$. 

\begin{remark}
	If $(\E,\A)$ is not regular, it is still possible that the IVP \eqref{eq:IVP} has a unique solution. In this case, the DDAE is called \emph{noncausal} and under some technical assumptions \cite{HaM15} 
	provides algorithms to transform \eqref{eq:DDAE2} such that the transformed pencil $(\tilde{\E},\tilde{\A})$ is regular.
\end{remark}

If $(\E,\A)$ is regular, then we can characterize the smoothness requirements for the inhomogeneity $\mosfDAE{i} = \B\mos{\state}{i-1} + \mos{f}{i}$ in \eqref{eq:methodOfSteps} for a classical solution to exist. This characterization is based on the Weierstra{\ss} canonical form (cf. \cite{Gan59b}). A more general form that is also valid for $\field = \mathbb{R}$ is the quasi-Weierstra{\ss} from, introduced in \cite{BerIT12}, which is presented here in a slightly different form.

\begin{theorem}[Quasi-Weierstra{\ss} form]
	\label{thm:quasiWf}
	The matrix pencil $(\E,\A)$ is regular if and only if there exist matrices $\WCFL,\WCFR\in\GL{\stateDim}{\field}$ such that
	\begin{equation}
		\label{eq:WCFDAE}
		\WCFL\E\WCFR= \begin{bmatrix}
			I_{\ndif} & 0\\
			0 & \N
		\end{bmatrix}\qquad\text{and}\qquad 
		\WCFL\A\WCFR = \begin{bmatrix}
			\J & 0 \\
			0 & I_{\nalg}\end{bmatrix},
	\end{equation}
	where $\N\in\field^{\nalg,\nalg}$ is a nilpotent matrix with index of nilpotency $\indEA$ and $\J\in\field^{\ndif,\ndif}$. If $\nalg>0$, we call $\indEA$ the index of the pencil $(\E,\A)$ and write $\ind(\E,\A)\vcentcolon= \indEA$. Otherwise we set $\ind(\E,\A)\vcentcolon=0$.
\end{theorem}

Applying the matrices $\WCFL$ and $\WCFR$ to the DAE \eqref{eq:methodOfSteps} implies a one-to-one correspondence between solutions of \eqref{eq:methodOfSteps} and solutions of
\begin{subequations}
	\label{eq:slowFast}
	\begin{align}
		\label{eq:slow}\dmos{\sstate}{i} &= \J\mos{\sstate}{i} + \smosfDAE{i},\\
		\label{eq:fast}\N\dmos{\fstate}{i} &= \mos{\fstate}{i} + \fmosfDAE{i},
	\end{align}	
\end{subequations}
with
\begin{displaymath}
	\begin{bmatrix}
		\mos{\sstate}{i}\\\mos{\fstate}{i}
	\end{bmatrix}\vcentcolon = \WCFR^{-1}\mos{\state}{i} \quad\text{and}\  
	\begin{bmatrix}
		\smosfDAE{i}\\\fmosfDAE{i}
	\end{bmatrix} \vcentcolon = \WCFL\mosfDAE{i}.
\end{displaymath}
While \eqref{eq:slow} is a standard ordinary differential equation (ODE) in $\mos{\sstate}{i}$ that can be solved with the Duhamel integral, the so called \emph{fast subsystem} \eqref{eq:fast} has the solution
\begin{equation}
	\label{eq:solFast}
	\mos{\fstate}{i} = -\sum_{k=0}^{\indEA-1} \N^k \derv{\fmosfDAE{i}}{k}
\end{equation}
and hence the function $\fmosfDAE{i}$ must be $\indEA$ times continuously differentiable for a classical solution to exist (cf. \cite{KunM06}). In addition, a consistent initial condition $\mos{\fstate}{i}(0)$ must satisfy \eqref{eq:solFast}. Similar to \cite{TanT10}, we define the matrices
\begin{align*}
	\Adiff &\vcentcolon= \WCFR\begin{bmatrix}
		\J & 0\\
		0 & 0
	\end{bmatrix}\WCFR^{-1}, & \Acon &\vcentcolon= \WCFR\begin{bmatrix}
		I_{\ndif} & 0\\ 
		0 & 0
	\end{bmatrix}\WCFR^{-1},\\
	\C{0} &\vcentcolon= \WCFR\begin{bmatrix}
		I_{\ndif} & 0\\
		0 & 0
	\end{bmatrix}\WCFL, & \C{k} &\vcentcolon= -\WCFR\begin{bmatrix}
		0&0\\
		0 & \N^{k-1}
	\end{bmatrix}\WCFL
\end{align*}
for $k=1,\ldots,\ind(\E,\A)$.

\begin{prop}
	\label{prop:underlyingODE}
	Assume that the DAE \eqref{eq:methodOfSteps} satisfies \Cref{ass:regular,ass:smoothness}. Then any classical solution $\mos{\state}{i}$ of \eqref{eq:methodOfSteps} fullfills the so called \emph{underlying ODE}
	\begin{equation}
		\label{eq:underlyingODE}
		\dmos{\state}{i} = \Adiff\mos{\state}{i} + \sum_{k=0}^{\ind(\E,\A)} \C{k}\derv{\mosfDAE{i}}{k}.
	\end{equation}
	Conversely, let $\mos{\state}{i}$ be a classical solution of \eqref{eq:underlyingODE}. Then $\mos{\state}{i}$ is a solution of \eqref{eq:methodOfSteps} if and only if there exists $s\in[0,\delay]$ such that $\mos{\state}{i}(s)$ satisfies
	\begin{equation}
		\label{eq:consistencyCon1}
		\mos{\state}{i}(s) = \Acon\mos{\state}{i}(s) + \sum_{k=1}^{\ind(\E,\A)} \C{k} \derv{\mosfDAE{i}}{k-1}(s).
	\end{equation}
\end{prop}

\begin{proof}
	Let $\mos{\state}{i}$ be a classical solution of \eqref{eq:methodOfSteps} and $\WCFL,\WCFR\in\GL{n}{\field}$ be matrices that satisfy \eqref{eq:WCFDAE} of the quasi-Weierstra{\ss} form and set $\indEA \vcentcolon= \ind(\E,\A)$. Differentiation of \eqref{eq:solFast} yields
	\begin{align*}
		\dmos{\state}{i} &= \WCFR\begin{bmatrix}
			\dmos{\sstate}{i}\\\dmos{\fstate}{i}
		\end{bmatrix} = \WCFR\begin{bmatrix}
			\J\mos{\sstate}{i} + \smosfDAE{i}\\
			-\sum_{k=0}^{\indEA-1}\N^k \derv{\fmosfDAE{i}}{k+1}
		\end{bmatrix}\\
		&= \WCFR\begin{bmatrix}
			\J & 0\\
			0 & 0
		\end{bmatrix}\begin{bmatrix}
			\mos{\sstate}{i}\\
			\mos{\fstate}{i}
		\end{bmatrix} + \WCFR\begin{bmatrix}
			I_{\ndif} & 0\\
			0 & 0
		\end{bmatrix}\begin{bmatrix}
			\smosfDAE{i}\\
			\fmosfDAE{i}
		\end{bmatrix} - \sum_{k=1}^{\indEA} \WCFR\begin{bmatrix}
			0 & 0\\
			0 & \N^{k-1}
		\end{bmatrix}\begin{bmatrix}
			\derv{\smosfDAE{i}}{k}\\
			\derv{\fmosfDAE{i}}{k}
		\end{bmatrix}\\
		&= \Adiff\mos{\state}{i} + \sum_{k=0}^{\indEA}\C{k}\derv{\mosfDAE{i}}{k}.
	\end{align*}
	Conversely, let $\mos{\state}{i}$ be a classical solution of \eqref{eq:underlyingODE}. Then there exists $\mos{\state}{i}(0)\in\field^{\stateDim}$ such that
	\begin{equation}
		\label{eq:variationOfConstants}
		\mos{\state}{i}(t) = \mathrm{e}^{\Adiff t}\mos{\state}{i}(0) + \int_0^t \mathrm{e}^{\Adiff(t-s)}\sum_{k=0}^{\indEA}\C{k}\derv{\mosfDAE{i}}{k}(s)\mathrm{d}s.
	\end{equation}
	Scaling \eqref{eq:variationOfConstants} from the left by $\WCFR^{-1}$ we obtain
	\begin{align*}
		\mos{\sstate}{i}(t) &= \mathrm{e}^{\J t}\sstate{i}(0) + \int_0^t \mathrm{e}^{\J(t-s)}\smosfDAE{i}(s)\mathrm{d}s\qquad\text{and}\\
		\mos{\fstate}{i}(t) &= \mos{\fstate}{i}(0) - \sum_{k=1}^{\indEA} \N^{k-1}\int_0^t \derv{\fmosfDAE{i}}{k}(s)\mathrm{d}s = \mos{\fstate}{i}(0) - \sum_{k=0}^{\indEA-1} \N^k \derv{\fmosfDAE{i}}{k}(t) + \sum_{k=0}^{\indEA-1} \N^k \derv{\fmosfDAE{i}}{k}(0).
	\end{align*}
	The condition \eqref{eq:consistencyCon1} implies the existence of $s\in[0,\delay]$ such that
	\begin{displaymath}
		\begin{bmatrix}
			\mos{\sstate}{i}(s)\\
			\mos{\fstate}{i}(s)
		\end{bmatrix} = \begin{bmatrix}
			\mos{\sstate}{i}(s)\\
			-\sum_{k=0}^{\indEA-1}\N^k\derv{\fmosfDAE{i}}{k}(s).
		\end{bmatrix}
	\end{displaymath}
	Together with \eqref{eq:solFast} this implies that $\mos{\state}{i}$ is a solution of \eqref{eq:methodOfSteps}.
\end{proof}

Setting $s=0$ in the previous proposition yields the following requirement for an initial condition to be consistent.
\begin{corollary}
	\label{cor:consistencyCondition}
	Assume that the DAE \eqref{eq:methodOfSteps} satisfies \Cref{ass:regular,ass:smoothness}. Then $\mos{\state}{i}(0)$ is consistent if and only if it satisfies the \emph{consistency condition}
	\begin{equation}
		\label{eq:consistencyCondition}
		\mos{\state}{i}(0) = \Acon\mos{\state}{i}(0) + \sum_{k=1}^{\ind(\E,\A)} \C{k}\derv{\mosfDAE{i}}{k-1}(0).
	\end{equation}
	In this case, the IVP \eqref{eq:sequenceDAEs} has a unique solution $\mos{\state}{i}\in\mathcal{C}^\infty([0,\delay],\field^{\stateDim})$.
\end{corollary}

In order to reformulate \eqref{eq:underlyingODE} in terms of the delayed argument we introduce the matrices $\Bk{k}\vcentcolon= \C{k}\B$ for $k=0,\ldots,\ind(\E,\A)$. This yields the DDE
\begin{equation}
	\label{eq:underlyingDDE}
	\dot{\state} = \Adiff\state + \sum_{k=0}^{\ind(\E,\A)} \left(\Bk{k}\shift{\delay}\derv{\state}{k} + \C{k}\derv{\f}{k}\right),
\end{equation}
which we call the the \emph{underlying DDE} for the DDAE \eqref{eq:DDAE2}.

In contrast to DAEs, a classical solution concept is not reasonable for the DDAE \eqref{eq:DDAE2}, because the identity
\begin{displaymath}
	\lim_{t \searrow 0} \dot{\state}(t) =\vcentcolon \dot{\state}(0^+) = \dot{\history}(0^-)\vcentcolon= \lim_{t \nearrow 0} \dot{\history}(t)
\end{displaymath}
is in general not satisfied and this discontinuity in the first derivative at $t=0$ may propagate over time \cite{BelZ03}. Instead, we use the following solution concept.

\begin{defn}[Solution concept]
	Assume that the DDAE \eqref{eq:DDAE} satisfies \Cref{ass:regular,ass:smoothness}. We call $\state\in\mathcal{C}(\timeInt,\field^{\stateDim})$ a \emph{solution} of \eqref{eq:IVP} if for all $i\in\mosIndexSet$ the restriction $\mos{\state}{i}$ of $\state$ as in \eqref{eq:restriction} is a solution of \eqref{eq:sequenceDAEs}. We call the history function $\history$ \emph{consistent} if the initial value problem \eqref{eq:IVP} has at least one solution.
\end{defn}

From \Cref{cor:consistencyCondition} and the discussion thereafter we immediately observe that a necessary condition for a history function $\history$ to be consistent is that is satisfies the equation
\begin{equation}
	\label{eq:admissibleHistory}
	\history(0) = \Acon\history(0) + \sum_{k=0}^{\ind(\E,\A)}\left( \Bk{k}\derv{\history}{k}(-\delay) + \C{k}\derv{\f}{k}(0)\right).
\end{equation}
Unfortunately, as \Cref{ex:advanced} suggests, this condition is not sufficient for consistency, which gives raise to the following definition.

\begin{defn}
	Assume that the IVP \eqref{eq:IVP} with history function $\history:[-\delay,0]\to\field^{\stateDim}$ satisfies \Cref{ass:regular,ass:smoothness}. Then $\history$ is called \emph{admissible} for the IVP \eqref{eq:IVP} if $\mos{\state}{1}(0) = \history(0)$ is consistent for the DAE
	\begin{displaymath}
		\E\dmos{\state}{1}(t) = \A\mos{\state}{1}(t) + \B\history(t-\delay) + \mosf{1}(t),
	\end{displaymath}
	i.\,e.\ $\history$ satisfies \eqref{eq:admissibleHistory}.
\end{defn}

For the analysis in the upcoming section, we introduce
\begin{equation}
	\label{eq:matricesWCF}
	\begin{bmatrix}
		\Bd\\
		\Ba
	\end{bmatrix} \vcentcolon= \WCFL\B,\qquad
	\begin{bmatrix}
		\Bdd & \Bda\\
		\Bad & \Baa	
	\end{bmatrix} \vcentcolon= \WCFL\B\WCFR, \qquad
	\begin{bmatrix}
		\ssf\\ \ff
	\end{bmatrix} \vcentcolon= \WCFR\f, \qquad\text{and}\qquad
	\begin{bmatrix}
		\shistory\\\fhistory
	\end{bmatrix} \vcentcolon= \WCFR\history,
\end{equation}
where $\WCFL,\WCFR\in\GL{\stateDim}{\field}$ are matrices that satisfy \eqref{eq:WCFDAE} from the quasi-Weierstra{\ss} form (\Cref{thm:quasiWf}) and we use the same block dimensions as in \eqref{eq:WCFDAE}. Applying the matrices $\WCFL,\WCFR$ to \eqref{eq:DDAE2} yields
\begin{subequations}
	\label{eq:WCF}
	\begin{align}
		\label{eq:slowSystem}\dot{\sstate} &= \J\sstate + \Bdd\shift{\delay}\sstate + \Bda\shift{\delay}\fstate + \ssf,\\
		\label{eq:fastSystem} \N\dot{\fstate} &= \fstate + \Bad\shift{\delay}\sstate + \Baa\shift{\delay}\fstate + \ff.
	\end{align}
\end{subequations}

\section{Discontinuity propagation}

In this section we derive a classification for the DDAE \eqref{eq:DDAE} in terms of the propagation of primary discontinuities of solutions of the IVP \eqref{eq:IVP}. Recall that for an admissible history function $\history:[-\delay,0]\to\field^{\stateDim}$, \Cref{ass:regular,ass:smoothness} guarantee that there exists a number $\maxMOS\in\mathbb{N}$ and a unique sequence $(\mos{\state}{i})_{i\in\{0,\ldots,\maxMOS\}}$ that satisfies \eqref{eq:sequenceDAEs} (cf. \Cref{cor:consistencyCondition}). Hence for any $i\in\{1,\ldots,\maxMOS\}$ we can define the level $\ell_i$ of the primary discontinuity as 
\begin{equation}
	\label{eq:primDiscLevelL}
	\ell_i \vcentcolon= \min_{\f\in\mathcal{C}^\infty(\timeInt,\field^{\stateDim})} ~ \min_{\substack{\history\in\mathcal{C}^\infty([-\delay,0],\field^{\stateDim})\\\history\ \text{admissible}}} ~ \max\{\ell\in\mathbb{N}_0\mid \mos{\state}{i}\ \text{solves \eqref{eq:sequenceDAEs} and } \derv{\mos{\state}{i}}{\ell}(0^+) = \derv{\mos{\state}{i-1}}{\ell}(\delay^-)\}.
\end{equation}
For $i>\maxMOS$ we formally set $\ell_i \vcentcolon= -\infty$. Note that this definition is independent of the specific choice of the inhomogeneity $\f$ and the history $\history$ and thus serves as the worst-case scenario. To simplify the computation of the numbers $\ell_i$ we observe the following, which is a generalization of \cite[Theorem~7.1]{HalL93}

\begin{lemma}
	\label{lem:smoothSol}
	Assume that the IVP \eqref{eq:IVP} with admissible history function $\history:[-\delay,0]\to\field^{\stateDim}$ satisfies \Cref{ass:regular,ass:smoothness}. Then the solution $\state$ of \eqref{eq:IVP} is continuously differentiable on $[-\delay,\delay]$ if and only if $\history$ satisfies
	\begin{equation}
		\label{eq:smoothnessCondition}
		\dot{\history}(0) = \Adiff\history(0) + \sum_{k=0}^{\ind(\E,\A)}\left(\Bk{k}\derv{\history}{k}(-\delay) + \C{k}\derv{\f}{k}(0)\right).	
	\end{equation}
\end{lemma}
\begin{proof}
	Since $\history$ is admissible, the initial condition $\mos{\state}{1}(0) = \history(0)$ is consistent and following \Cref{cor:consistencyCondition} the solution $\state$ exists on $[-\delay,\delay]$. Thus, it is sufficient to check the point $t=0$. Using \Cref{prop:underlyingODE} we can consider \eqref{eq:underlyingODE} and thus obtain
	\begin{align*}
		\dmos{\state}{1}(0^+) &= \Adiff\mos{\state}{1}(0) + \sum_{k=0}^{\ind(\E,\A)}\left(\Bk{k} \derv{\mos{\state}{0}}{k}(0) + \C{k}\derv{\mosf{1}}{k}(0)\right)\\
		&= \Adiff\history(0) + \sum_{k=0}^{\ind(\E,\A)}\left(\Bk{k}\derv{\history}{k}(-\delay) + \C{k}\derv{\mosf{1}}{k}(0)\right)
	\end{align*}
	and hence $\state$ is continuously differentiable on $[-\delay,\delay]$ if and only if $\history$ satisfies \eqref{eq:smoothnessCondition}.
\end{proof}
Since we require $\history\in\mathcal{C}^{\infty}([-\delay,0],\field^{\stateDim})$ to be admissible we immediately obtain $\ell_1\geq 0$. On the other hand assume that we have given the values $\history(0)$ and $\derv{\history}{k}(-\delay)$ for $k=0,\ldots,\nu$ such that $\history$ is admissible. Then we can always construct (via Hermite interpolation) $\history$ in such a way that \eqref{eq:smoothnessCondition} is not satisfied and hence $\ell_1\leq 0$, which yields $\ell_1 = 0$. Thus, the questions about propagation of discontinuities can be rephrased as whether there exists $k\in\mathbb{N}$ with $\ell_k>0$ (i.\,e.\ the solution becomes smoother), or there exists $k\in\mathbb{N}$ with $\ell_k = -\infty$ (i.\,e.\ the solution becomes less smooth), or if $\ell_i = \ell_1$ for all $i\in\mathbb{N}$. Note that the smoothing may not start immediately (i.\,e.\ we cannot ask for $\ell_1=1$), as the following example suggests.

\begin{example}
	\label{ex:slowSmoothing}
	Consider the DDAE given by $\field = \mathbb{R}$, $\stateDim =2$, $\f\equiv 0$, $\delay=1$, and
	\begin{displaymath}
		\E = \begin{bmatrix}
			1 & 0\\
			0 & 0
		\end{bmatrix},\qquad \A = \begin{bmatrix}
			0 & 0\\
			0 & 1
		\end{bmatrix},\qquad \B = \begin{bmatrix}
			0 & 1\\
			-1 & 0
		\end{bmatrix},\qquad
		\history(t) = \begin{bmatrix}
			t,\\
			-1
		\end{bmatrix}.
	\end{displaymath}
	Since $(\E,\A)$ is already in Weierstra\ss{} form, it is easy to see that the DDAE corresponds to the DDE
	\begin{equation}
		\label{eq:retardedSubequation}
		\dot{\sstate}(t) = \sstate(t-2\delay)
	\end{equation}
	with coupled equation $\fstate(t) = \sstate(t-\delay)$. Straight forward calculations show that $\ell_1\leq 0$ (using the specified history function $\history$) and $\ell_1\geq 0$ implying $\ell_1=0$. On the other, \eqref{eq:retardedSubequation} is a scalar delay equation and it is well-known, that the solution is continuously differentiable at $t=2\delay$, thus we have $\ell_2\geq 1$.
\end{example}

\begin{defn}[Classification]
	\label{def:classification}
	Consider the DDAE \eqref{eq:DDAE} on the interval $\timeInt = [0,\maxMOS\delay]$, set $\mosIndexSet \vcentcolon= \{1,\ldots,\maxMOS\}$, and suppose that \eqref{eq:DDAE} satisfies \Cref{ass:regular,ass:smoothness}. We say that \eqref{eq:DDAE} is of 
	\begin{itemize}
		\item \retarded{} if there exists $j\in\mosIndexSet$, $j>1$ such that $\ell_j=1$ and $\ell_i=0$ for $i<j$,
		\item \neutral{} if $\ell_i = 0$ for all $i\in\mosIndexSet$, and
		\item \advanced{} if there exists $j\in\mosIndexSet$, $j>1$ such that $\ell_j=-\infty$ and $\ell_i=0$ for $i<j$.
	\end{itemize}
\end{defn}

In the following, we analyze in detail the DDAE \eqref{eq:DDAE} and derive conditions for the matrices $\E,\A$, and $\B$, from which the type can be determined. Before we analyze the general DDAE case we focus on the case of $\ind(\E,\A)\leq1$, i.\,e., the system is a pure DDE or $N=0$ in \eqref{eq:fast}. Note that this case includes DDEs of the form
\begin{equation}
	\label{eq:neutralDDE}
	\dot{\hat{\state}}(t) = \hat{\A}\hat{\state}(t) + \hat{\B}\hat{\state}(t-\delay) + \hat{\D}\dot{\hat{\state}}(t-\delay) + \hat{\f}(t),
\end{equation}
with arbitrary matrices $\hat{\A},\hat{\B},\hat{\D}\in\field^{\stateDim,\stateDim}$. 

If $\ind(\E,\A) = 0$, then the matrix $\E$ is nonsingular and the DDAE is of the form
\begin{equation}
	\label{eq:retardedDDE}
	\dot{\state}(t) = \E^{-1}\A\state(t) + \E^{-1}\B\state(t-\delay) + \E^{-1}\f(t)
\end{equation}
and the ODE solution formula together with \Cref{lem:smoothSol} directly implies $\ell_1=1$, i.\,e.\ \eqref{eq:retardedDDE} is of \retarded{}.

\begin{theorem}
	\label{thm:index1}
	Consider the DDAE \eqref{eq:DDAE} on the interval $\timeInt = [0,\maxMOS\delay]$ and suppose that \Cref{ass:regular,ass:smoothness} hold. If $\ind(\E,\A)=1$, then \eqref{eq:DDAE} is of \retarded{} if and only if $\Baa$ in \eqref{eq:matricesWCF} is nilpotent with index of nilpotency $\indEA_{\B}$ and furthermore we have $\indEA_{\B}\leq\maxMOS-1$. 
\end{theorem}

\begin{proof}
	Let $\WCFL,\WCFR\in\GL{n}{\field}$ be matrices that transform \eqref{eq:DDAE} into quasi-Weierstra{\ss} form \eqref{eq:WCF}. Applying the method of steps yields
	\begin{displaymath}
		\dmos{\sstate}{i+1} = \J\mos{\sstate}{i+1} + \Bdd\mos{\sstate}{i} + \Bda\mos{\sstate}{i} + \smosf{i+1}\qquad\text{and}\qquad
		\mos{\fstate}{i+1} = -\Bad\mos{\sstate}{i} - \Baa\mos{\fstate}{i} - \fmosf{i+1}.
	\end{displaymath}
	Since $\ell_1=0$ we have
	\begin{align*}
		\mos{\fstate}{1}(\delay) &= -\Bad\mos{\sstate}{0}(\delay) - \Baa\mos{\fstate}{0}(\delay) - \fmosf{1}(\delay)\\
		&= -\Bad\mos{\sstate}{1}(0) - \Baa\mos{\fstate}{1}(0) - \fmosf{2}(0) = \mos{\fstate}{2}(0)
	\end{align*}
	and thus $\ell_2\geq 0$. By induction we conclude $\ell_i\geq 0$ for $i\in\mosIndexSet$. Moreover, we have
	\begin{align*}
		\dmos{\fstate}{i+1} &= -\Bad\dmos{\sstate}{i} - \Baa\dmos{\fstate}{i} - \dmos{\ff}{i+1} \\
		&= -\Bad\left(\J\mos{\sstate}{i} + \Bdd\mos{\sstate}{i-1} + \Bda\mos{\fstate}{i-1} + \smosf{i}\right) - \Baa\dmos{\fstate}{i} - \dmos{\ff}{i+1}
	\end{align*}
	which implies $\dmos{\fstate}{i+1}(0^+)-\dmos{\fstate}{i}(\delay^-) = \Baa\left(\dmos{\fstate}{i-1}(\delay^-)-\dmos{\fstate}{i}(0^+)\right)$ holds. By induction we have 
	\begin{displaymath}
		\dmos{\fstate}{i+1}(0^+)-\dmos{\fstate}{i}(\delay^-) = (-1)^i \Baa^i\left(\dmos{\fstate}{1}(0^+)-\dot{\fhistory}(0^-)\right)\qquad\text{for } i=1,\ldots,\maxMOS-1.
	\end{displaymath}
	Thus $\ell_{i+1} \geq 1$ holds if and only if $\Baa^i=0$. 
\end{proof}

Applying \Cref{thm:index1} to the DDAE in \Cref{ex:slowSmoothing} shows that this DDAE is of \retarded, since it is already in quasi-Weierstra{\ss} form with $\Baa=0$. Conversely, if the DDAE \eqref{eq:DDAE} with $\ind(\E,\A)=1$ is of \retarded, then the index of nilpotency indicates the number of delays present in the system. More precisely, we have the following result.

\begin{corollary}
	\label{cor:hiddenDelays}
	Suppose that the DDAE \eqref{eq:DDAE} satisfies \Cref{ass:regular,ass:smoothness} and is of \retarded{} with $\ind(\E,\A)\leq 1$. Furthermore let $\indEA_{\B}$ denote the index of nilpotency of $\Baa$ if $\nalg>0$ and $\indEA_{\B}=0$ otherwise. Then there exists matrices $\Dk{k}\in\field^{\ndif,\ndif}$ ($k=0,\ldots,\indEA_{\B}$) and an inhomogeneity $\vartheta$ such that the solution $\sstate$ of \eqref{eq:slowSystem} is a solution of the inital value problem
	\begin{subequations}
		\label{eq:correspondingRetardedIVP}
		\begin{align}
			\label{eq:correspondingDDE}
			\dot{\z}(t) &= \J\z + \sum_{k=0}^{\indEA_{\B}} \Dk{k} \z(t-(k+1)\delay) + \vartheta(t) && \mathrm{for}\ t\in(\indEA_{\B}\delay,\tf],\\
			\z(t) &= \sstate(t), && \mathrm{for }\ t\in[-\delay,\indEA_{\B}\delay].
		\end{align}
	\end{subequations}
\end{corollary}

\begin{proof}
	The result is trivial for $\ind(\E,\A)=0$, i.\,e., assume $\ind(\E,\A) = 1$, which implies that $\N=0$ in \eqref{eq:WCF}. Let $\indicator{(t_0,t_1]}$ denote the characteristic function for the interval $(t_0,t_1]$, i.\,e.\
	\begin{displaymath}
		\indicator{(t_0,t_1]}(t) = \begin{cases} 1, & \text{if}\ t\in(t_0,t_1],\\
		0, & \text{otherwise}.
		\end{cases}
	\end{displaymath}
	Combination of the fast subsystem \eqref{eq:fastSystem} and the initial condition yields
	\begin{equation}
		\label{eq:solAlgebraic}
		(I_{\nalg} + \Baa\indicator{(\delay,\tf]}\shift{\delay})\fstate = -\Ba\indicator{(0,\delay]} \shift{\delay}\history - \Bad\indicator{(\delay,\tf]}\shift{\delay}\sstate - \ff.
	\end{equation}
	By induction we obtain $(\indicator{(\delay,\tf]}(t)\shift{\delay})^k = \indicator{(k\delay,\tf]}(t)\shift{k\delay}$ and from $\Baa^{\indEA_{\B}}=0$ we deduce
	\begin{displaymath}
		\left(\sum_{k=0}^{\indEA_{\B}-1}(-1)^k\left(\Baa\indicator{(\delay,\tf]}\shift{\delay}\right)^k\right) \left(I_{\nalg} + \Baa\indicator{(\delay,\tf]}\shift{\delay}\right) = I_{\nalg}
	\end{displaymath}
	such that $\fstate$ in \eqref{eq:solAlgebraic} is given by
	\begin{align*}
		\fstate &= \sum_{k=0}^{\indEA_{\B}-1}(-1)^{k+1} \left(\Baa\indicator{(\delay,\tf]} \shift{\delay}\right)^k \left(\Ba\indicator{(0,\delay]}\shift{\delay}\history + \Bad\indicator{(\delay,\tf]}\shift{\delay}\sstate + \ff\right)\\
		&= \sum_{k=0}^{\indEA_{\B}-1}(-1)^{k+1}\Baa^k \left(\Ba\indicator{(k\delay,(k+1)\delay]}\shift{(k+1)\delay} \history + \Bad\indicator{(k\delay,\tf]}\shift{(k+1)\delay} \sstate + \indicator{(k\delay,\tf]}\shift{k\delay} \ff\right).
	\end{align*}
	Inserting this identity in \eqref{eq:slowSystem} and introducing for $k=1,\ldots,\indEA_{\B}$ the matrices
	\begin{displaymath}
		\Dk{0} \vcentcolon= \Bdd,\qquad \Dk{k}\vcentcolon= (-1)^k \Bda\Baa^{k-1}\Bad
	\end{displaymath}
implies that the solution $\sstate$ of \eqref{eq:slowSystem} is a solution of the IVP \eqref{eq:correspondingRetardedIVP}, where $\vartheta$ is given by
	\begin{displaymath}
		\vartheta(t) \vcentcolon= \ssf(t) + \sum_{k=0}^{\indEA_{\B}-1}(-1)^{k+1}\Bda\Baa^k \ff(t-(k+1)\delay).\qedhere
	\end{displaymath}
\end{proof}

\begin{example}
	Consider the DDE \eqref{eq:neutralDDE}.	Introducing the shifted variable $y(t) = \hat{\state}(t-\delay)$ yields the DDAE 
	\begin{displaymath}
		\begin{bmatrix}
			-I_{\stateDim} & -\D\\
			0 & 0
		\end{bmatrix}\dot{\state}(t) = \begin{bmatrix}
			\A & 0\\
			0 & I_{\stateDim}
		\end{bmatrix}\state(t) + \begin{bmatrix}
			\B & 0\\
			-I_{\stateDim} & 0
		\end{bmatrix}\state(t-\delay) + \begin{bmatrix}
			\f\\0
		\end{bmatrix}.
	\end{displaymath}		
	The matrices $\WCFL \vcentcolon= \left[\begin{smallmatrix}
			I_{\stateDim} & -\A\D\\
			0 & I_{\stateDim}
		\end{smallmatrix}\right]$ and $\WCFR \vcentcolon= \left[\begin{smallmatrix}
			I_{\stateDim} & \D\\
			0 & I_{\stateDim}
		\end{smallmatrix}\right]$	transform the DDAE to quasi-Weierstra{\ss} form given by 
	\begin{displaymath}
		\begin{bmatrix}
			I_{\stateDim} & 0\\0 & 0
		\end{bmatrix}\begin{bmatrix}
			\dot{\sstate}(t)\\\dot{\fstate}(t)
		\end{bmatrix} = \begin{bmatrix}
			\A & 0\\0 & I_{\stateDim}
		\end{bmatrix}\begin{bmatrix}
			\sstate(t)\\\fstate(t)
		\end{bmatrix} + \begin{bmatrix}
			\B+\A\D & (\B+\A\D)\D\\ -I_{\stateDim} & -\D
		\end{bmatrix}\begin{bmatrix}
			\sstate(t-\delay)\\\fstate(t-\delay)
		\end{bmatrix} + \begin{bmatrix}
			\f(t)\\0
		\end{bmatrix}.
	\end{displaymath}
	Hence, the DDE \eqref{eq:neutralDDE} is of \retarded{} if and only if $\D$ is nilpotent. In this case, the corresponding retarded equation \eqref{eq:correspondingDDE} is given by
	\begin{displaymath}
		\dot{\z}(t) = \A\z(t) + (\B+\A\D)\z(t-\delay) + \sum_{k=1}^{\indEA_{\D}-1}(-1)^k (\B+\A\D)\D^k  \z(t-(k+1)\delay) + g(t),
	\end{displaymath}
	where $\indEA_{\D}$ is the index of nilpotency of $\D$.
\end{example}

\begin{remark}
	\label{rem:stability}
	The delay equation \eqref{eq:correspondingRetardedIVP} of \Cref{cor:hiddenDelays} may be used to determine whether the DDAE \eqref{eq:neutralDDE} is stable (which can be done for example via DDE-biftool \cite{EngLR02,SieELSR14}). Note that this provides an alternative way to the theory outlined in \cite{DuLMT13, DuLM13}.
\end{remark}

For the analysis of the general DDAE case with arbitrary index we use the following preliminary result.

\begin{prop}
	\label{prop:smoothSol2}
	Suppose that the IVP \eqref{eq:IVP} satisfies \Cref{ass:regular,ass:smoothness} and let $\WCFL,\WCFR\in\GL{\stateDim}{\field}$ be matrices that transform $(\E,\A)$ to quasi-Weierstra{\ss} form \eqref{eq:WCF}. Then for any $m\in\mathbb{N}$ and any $\tilde{\sstate}\in\field^{\ndif}$ there exists an admissible history function $\history = \WCFR^{-1}\smash{\begin{bmatrix}
		\shistory^T & \fhistory^T
\end{bmatrix}^T}$ that is analytic and satisfies
	\begin{subequations}
		\label{eq:conditionsSmoothPsi}
		\begin{align}
			\label{eq:conditionsSmoothPsi1}\derv{\shistory}{j}(0) &= \derv{\mos{\sstate}{1}}{j}(0+) &\text{for } j&=0,1,\ldots,m-1,\\
			\label{eq:conditionsSmoothPsi2}\derv{\fhistory}{j}(0) &= \derv{\mos{\fstate}{1}}{j}(0+) &\quad\text{for } j&=0,1,\ldots,m, &\text{and}\\
			\label{eq:conditionsSmoothPsi3}\tilde{\sstate} &= \derv{\shistory}{m}(0) - \derv{\mos{\sstate}{1}}{m}(0^+).
		\end{align}	 
	\end{subequations}
	Similarly for any $m\in\mathbb{N}$ and $\tilde{\fstate}\in\field^{\nalg}$ there exists an admissible and analytic history function $\history = \WCFR^{-1}\smash{\begin{bmatrix}
		\shistory^T & \fhistory^T
\end{bmatrix}^T}$ that satisfies
	\begin{subequations}	
		\label{eq:conditionsSmoothEta}	
		\begin{align}
			\label{eq:conditionsSmoothEta1}\derv{\shistory}{j}(0) &= \derv{\mos{\sstate}{1}}{j}(0+) &\text{for}\ j&=0,1,\ldots,m,\\
			\label{eq:conditionsSmoothEta2}\derv{\fhistory}{j}(0) &= \derv{\mos{\fstate}{1}}{j}(0+) &\quad\text{for}\ j&=0,1,\ldots,m-1, &\text{and}\\
			\label{eq:conditionsSmoothEta3}\tilde{\fstate} &= \derv{\fhistory}{m}(0) - \derv{\mos{\fstate}{1}}{m}(0^+).
		\end{align}
	\end{subequations}
\end{prop}

\begin{proof}
	Let $m\in\mathbb{N}$. By induction, \Cref{lem:smoothSol} implies that the solution $\state$ of the IVP \eqref{eq:IVP} is $m$ times continuously differentiable on $[-\delay,\delay]$ if and only if $\history$ satisfies
	\begin{equation}
		\label{eq:smoothTransition}
		\derv{\history}{j+1}(0) = \Adiff\derv{\history}{j}(0) + \sum_{k=0}^{\ind(\E,\A)-1}\left(\Bk{k}\derv{\history}{k+j}(-\delay) + \C{k}\derv{\f}{k+j}(0)\right)
	\end{equation}
	for $j=0,1,\ldots,m-1$. Multiply \eqref{eq:smoothTransition} from the left by $\WCFR^{-1}$ to obtain
	\begin{subequations}
		\label{eq:smoothnessConditions}
		\begin{align}
			\label{eq:smoothnessPsi}
			\derv{\shistory}{j+1}(0) &= \J\derv{\shistory}{j}(0) + \Bdd\derv{\shistory}{j}(-\delay) + \Bda\derv{\fhistory}{j}(-\delay) + \derv{\ssf}{j}(0),\\
			\label{eq:smoothnessEta}
			\derv{\fhistory}{j+1}(0) &= -\sum_{k=0}^{\ind(\E,\A)-1}\N^k\left(\Bad\derv{\shistory}{k+j+1}(-\delay) + \Baa\derv{\fhistory}{k+j+1}(-\delay) + \derv{\ff}{k+j+1}(0)\right)
		\end{align}
	\end{subequations}
	for $j=0,\ldots,m-1$. We then can proceed as follows to construct $\shistory$ and $\fhistory$ that satisfy conditions~\eqref{eq:conditionsSmoothPsi}. Choose any value for $\derv{\shistory}{k}(-\delay)$ and $\derv{\fhistory}{k}(-\delay)$ for $k=0,\ldots,\ind(\E,\A)+m$. This fixes the values $\derv{\fhistory}{k}(0)$ for $k=0,\ldots,m$ by \eqref{eq:smoothnessEta}. For an arbitrary $\shistory(0)$, set $\derv{\shistory}{j+1}(0)$ according to \eqref{eq:smoothnessPsi} for $j=0,\ldots,m-2$. Finally, set
	\begin{displaymath}
		\derv{\shistory}{m}(0) = \tilde{\sstate} - \left(\J\derv{\shistory}{j}(0) + \Bdd\derv{\fhistory}{j}(-\delay) + \Bda\derv{\fhistory}{j}(-\delay) + \derv{\ssf}{j}(0)\right).
	\end{displaymath}
	The desired history functions are then given via Hermite interpolation. The construction for $\shistory$ and $\fhistory$ that satisfy \eqref{eq:conditionsSmoothEta} proceeds analogously.
\end{proof}

Applying the method of steps and the solution formula \eqref{eq:solFast} for the fast subsystem yields
\begin{equation}
	\label{eq:fastSystemMethodSteps}
	\mos{\fstate}{i+1} = -\sum_{k=0}^{\ind(\E,\A)-1}\N^k\left(\frac{\mathrm{d}}{\mathrm{d}t}\right)^k\left(\Bad\mos{\sstate}{i} + \Baa\mos{\fstate}{i} + \fmosf{i+1}\right).
\end{equation}
Since \Cref{ass:smoothness} implies that all functions are sufficiently smooth we obtain
\begin{align*}
	\mos{\fstate}{2}(0^+)-\mos{\fstate}{1}(\delay^-) &= \sum_{k=0}^{\ind(\E,\A)-1}\N^k\left(\Bad\left(\derv{\shistory}{k}(0) - \derv{\mos{\sstate}{1}}{k}(0^+)\right) + \Baa\left(\derv{\fhistory}{k}(0)-\derv{\mos{\fstate}{1}}{k}(0^+)\right)\right)\\
	&= \sum_{k=0}^{\ind(\E,\A)-1} \N^k\Ba\WCFR\begin{bmatrix}
		\derv{\shistory}{k}(0) - \derv{\mos{\sstate}{1}}{k}(0^+)\\
		\derv{\fhistory}{k}(0) - \derv{\mos{\fstate}{1}}{k}(0^+)
	\end{bmatrix}\\
	&= \sum_{k=1}^{\ind(\E,\A)-1} \N^k\Ba\WCFR\begin{bmatrix}
		\derv{\shistory}{k}(0) - \derv{\mos{\sstate}{1}}{k}(0^+)\\
		\derv{\fhistory}{k}(0) - \derv{\mos{\fstate}{1}}{k}(0^+)
	\end{bmatrix},
\end{align*}
where the last identity follows from the fact the $\history$ is assumed to be admissible. \Cref{prop:smoothSol2} implies that \eqref{eq:DDAE} is of \advanced{} if there exists $k\in\{1,\ldots,\ind(\E,\A)-1\}$ such that $\N^k\Ba\neq0$. Assume conversely that $\N\Ba=0$. In this case \eqref{eq:fastSystemMethodSteps} is given by
\begin{displaymath}
	\mos{\fstate}{i+1} = -\Bad\mos{\sstate}{i} - \Baa\mos{\fstate}{i} - \sum_{k=0}^{\ind(\E,\A)-1} \N^k \derv{\fmosf{i+1}}{k},
\end{displaymath}
which implies $\ell_i\geq0$. Together with \Cref{thm:index1}, this proofs the following theorem.

\begin{theorem}
	\label{thm:classificationDDAE}
	Consider the DDAE \eqref{eq:DDAE} on the interval $\timeInt = [0,\maxMOS\delay]$ and suppose that \Cref{ass:regular,ass:smoothness} hold. Let $\N$, $\Ba$ and $\Baa$ be the matrices that are associated with the quasi-Weierstra{\ss} form \eqref{eq:WCF}. Then \eqref{eq:DDAE} is of
	\begin{itemize}
		\item \retarded{} if $\N\Ba=0$ and $\Baa$ is nilpotent with nilpotency index $\indEA_{\B}<\maxMOS$,
		\item \advanced{} if there exists $k\in\mathbb{N}$ such that $\N^k\Ba\neq0$, and
		\item \neutral{} otherwise.
	\end{itemize}
\end{theorem}

\begin{example}
	Introducing the shifted variable $\z(t) = \state(t-\delay)$ shows that the DDAE associated with
	\begin{equation}
		\label{eq:advancedSystem}
		\state(t) = \B\state(t-\delay) + \D\dot{\state}(t-\delay) + \f(t)
	\end{equation}
	is of \advanced{} if and only if $\D\neq0$.
\end{example}

\begin{remark}
	\label{rem:smoothing}
	Checking the proof of \Cref{cor:hiddenDelays}, we immediately infer from \Cref{thm:classificationDDAE} that \Cref{cor:hiddenDelays} is also true for arbitrary index $\ind(\E,\A)$. As a consequence, if the DDAE \eqref{eq:DDAE} is of smoothing type, then there exists a sequence $j_k\in\mathbb{N}$ such that $\ell_{j_k} = k$ and hence the solution becomes arbitrary smooth over time, which justifies the name \retarded{}.
\end{remark}

A common approach to analyze the (exponential) stability of the DDAE \eqref{eq:DDAE} is to compute the spectral abscissa, which is defined as
\begin{displaymath}
		\alpha(\E,\A,\D) = \sup \{\mathrm{Re}(\lambda) \mid \det(\lambda\E-\A-\exp(-\lambda\delay)\D) = 0\}.
\end{displaymath}
Surprisingly, the condition $\alpha(\E,\A,\D)<0$ is not sufficient for a DDAE to be exponentially stable \cite{DuLMT13}. However, based on the new classification we have the following result.

\begin{corollary}
	\label{cor:stability}
	Suppose that the DDAE \eqref{eq:DDAE} is not of \advanced{}. Then the DDAE \eqref{eq:DDAE} is exponentially stable if and only if $\alpha(\E,\A,\D)<0$.
\end{corollary}

\begin{proof}
	Since the DDAE \eqref{eq:DDAE} is not of \advanced{}, we have $\N\Ba = 0$. The result follows directly from \cite[Proposition~3.4 and Theorem~3.4]{DuLMT13}.
\end{proof}

Note that we refrain from using the terminology retarded, neutral, and advanced in \Cref{def:classification}, although these terms are widely used in the delay literature \cite{BelC63,BelZ03,HaM15,HalL93}. The reason for this is, that in the classical definition in \cite{BelC63}, a retarded DDE becomes advanced if it is solved backwards in time, an advanced equation becomes retarded and a neutral equation stays neutral. For the classification introduced in \Cref{def:classification} this is however not true. To see this, we introduce the new variable $\xi(t-\tau) = \state(-t)$ such that \eqref{eq:DDAE} transforms to
\begin{equation*}
	\label{eq:backwardSystem1}
	\E\dot{\xi}(t-\delay) = -\B\xi(t) - \A\xi(t-\delay) - f(-t).
\end{equation*}
This leads to the following definition.

\begin{defn}
	Consider the DDAE \eqref{eq:DDAE}. Define
	\begin{displaymath}
		\bE \vcentcolon= \begin{bmatrix}
			0 & \E\\
			0 & 0
		\end{bmatrix}\in\field^{2\stateDim,2\stateDim},\qquad
		\bA \vcentcolon= \begin{bmatrix}
			-\B & 0\\
			0 & I_{\stateDim}
		\end{bmatrix}\in\field^{2\stateDim,2\stateDim},\qquad
		\bB \vcentcolon= \begin{bmatrix}
			-\A & 0\\
			-I_{\stateDim} & 0
		\end{bmatrix}\in\field^{2\stateDim,2\stateDim}.
	\end{displaymath}
	Then we call the DDAE
	\begin{equation}
		\label{eq:backwardSystem}
		\bE\dot{\bstate}(t) = \bA\bstate(t) + \bB\bstate(t-\delay) + \bff(t)
	\end{equation}
	with $\bff:\timeInt\to\field^{2\stateDim}$ the \emph{backward system} for the DDAE \eqref{eq:DDAE}.
\end{defn}

Note that the backward system satisfies \Cref{ass:regular} if and only if $\det(\B)\neq 0$. In this case, we can transform the backward system \eqref{eq:backwardSystem} to quasi-Weierstra{\ss} form via the matrices
\begin{displaymath}
	\WCFL = \begin{bmatrix}
		-\B^{-1} & 0\\
		0 & I_{\stateDim}
	\end{bmatrix} \qquad\text{and}\qquad\WCFR = I_{2\stateDim}.
\end{displaymath}
In particular, we have
\begin{displaymath}
	(\WCFL\bE\WCFR)(\WCFL\bB\WCFR) = \begin{bmatrix}
		0 & -\B^{-1}\E\\
		0 & 0
	\end{bmatrix}\begin{bmatrix}
		\B^{-1}\A & 0\\
		-I_{\stateDim} & 0
	\end{bmatrix} = \begin{bmatrix}
		-\B^{-1}\E & 0\\
		0 & 0
	\end{bmatrix}.
\end{displaymath}
Thus \Cref{thm:classificationDDAE} implies that $\E=0$ is a necessary condition for the backward system \eqref{eq:backwardSystem} to be of \retarded{} or \neutral{}, which implies that the DDAE \eqref{eq:DDAE} cannot be of \advanced{}. 
\begin{example}
	\label{ex:backwardSystem}
	Consider the DDAE given by $\field = \mathbb{R}$, $\stateDim =2$, $\f\equiv 0$, $\delay=1$, and
	\begin{displaymath}
		\E = \begin{bmatrix}
			0 & 1\\
			0 & 0
		\end{bmatrix},\qquad \A = \begin{bmatrix}
			1 & 0\\
			0 & 1
		\end{bmatrix},\qquad \B = \begin{bmatrix}
			1 & 1\\
			0 & 1
		\end{bmatrix}.
	\end{displaymath}
	Since $(\E,\A)$ is already in Weierstra\ss{} form and $\E\B\neq 0$, \Cref{thm:classificationDDAE} implies that the DDAE is of \advanced{}. Since $\E\neq 0$ also the backward system is of \advanced{}.
\end{example}

Let us mention that if $\det(B)=0$, then the method of steps \eqref{eq:sequenceDAEs} cannot be used to determine the solution of the backward system. Instead, one may use the shift-index concept defined in \cite{HaMS14,HaM15} to make the pencil $(\bE,\bA)$ regular.

\section{Impact of splicing conditions}
In the previous section we have established algebraic criteria to check whether a discontinuity in the derivative of $\dot{\state}$ at $t=0$ is smoothed out, is propagated to $t=\delay$ or is amplified in the sense that $\state$ becomes discontinuous at $t=\delay$. While the definition of \neutral{} is valid for all integer multiples of the delay time, the definitions of \retarded{} and \advanced{} are based on single time points and hence the question whether the (de-)smoothing continues is imminent. For DDAEs of \retarded{}, this can be answered positively (see \Cref{rem:smoothing}). For DDAEs of \advanced{} the question can be rephrased as follows: If we restrict the set of admissible history functions such that the \emph{splicing condition} (cf. \cite{BelZ03})
\begin{equation}
	\label{eq:splicingCondition}
	\derv{\history}{k}(0) = \derv{\state}{k}(0^+)\qquad\text{for}\ k=0,\ldots,\kappa
\end{equation}
is satisfied for some $\kappa\in\mathbb{N}$, is there an integer $j\in\mathbb{N}$ such that the inital condition
\begin{displaymath}
	\mos{\state}{j}(0) = \mos{\state}{j-1}(\delay)
\end{displaymath}
is not consistent for the DAE \eqref{eq:sequenceDAEs}. For DDAEs of \neutral{}, we expect that the smoothness at integer multiples of the delay time stays invariant. This is indeed the case as the following result shows.

\begin{lemma}
	\label{lem:moreRegularityNeutral}
	Suppose that the DDAE \eqref{eq:DDAE} is of \neutral{} and the history function $\history\in\mathcal{C}^{\infty}(\timeInt,\field^{\stateDim})$ satisfies the splicing condition \eqref{eq:splicingCondition}. Then
	\begin{displaymath}
		\derv{\mos{\state}{i}}{k}(0^+) = \derv{\mos{\state}{i-1}}{k}(\delay^-)\qquad\text{for all}\ i\in\mathbb{N},\ k=0,\ldots,\kappa.
	\end{displaymath}
\end{lemma}

\begin{proof}
	Since \eqref{eq:DDAE} is of \neutral{}, we have $\N\Ba=0$ in \eqref{eq:WCF} according to \Cref{thm:classificationDDAE}. It suffices to show that
	\begin{displaymath}
		\derv{\mos{\state}{2}}{j}(0^+) = \derv{\mos{\state}{1}}{j}(\delay^-)\qquad\text{for all}\ j=0,\ldots,\kappa.
	\end{displaymath}
	Since $\history$ is admissible and the DDAE is of \neutral{}, equation \eqref{eq:slowSystem} implies that
	\begin{align*}
		\dmos{\sstate}{2}(0^+)-\dmos{\sstate}{1}(\delay^-) = \J\left(\mos{\sstate}{2}(0)-\mos{\sstate}{1}(\delay)\right) + \Bd\left(\mos{\state}{1}(0) - \history(0)\right) = 0.
	\end{align*}
	Iteratively, we obtain 
	\begin{align*}
		\derv{\mos{\sstate}{2}}{k+1}(0^+) - \derv{\mos{\sstate}{1}}{k+1}(\delay^-) = \J\left(\derv{\mos{\sstate}{2}}{k}(0^+)-\derv{\mos{\sstate}{1}}{k}(\delay^-)\right) + \Bd\left(\derv{\mos{\state}{1}}{k}(0^+) - \derv{\history}{k}(0)\right) = 0
	\end{align*}
	for $k=2,\ldots,\kappa$. For the fast system \eqref{eq:fastSystem} we infer directly
	\begin{displaymath}
		\derv{\mos{\fstate}{2}}{k}(0^+)-\derv{\mos{\fstate}{1}}{k}(\delay^-) = \Ba\left(\derv{\history}{k}(0)-\derv{\mos{\state}{1}}{k}(0^+)\right) = 0
	\end{displaymath}
	for $k=0,1,\ldots,\kappa$, which completes the proof.
\end{proof}

Note that \Cref{lem:moreRegularityNeutral} guarantees that the solution of the DDAE is at least as smooth as the initial transition from the history function to the solution. Conversely, assume that the Jordan canonical form of $\Baa$ exists and let $\tilde{\fstate}\in\field^{\nalg}\setminus\{0\}$ be an eigenvector of $\Baa$ for the eigenvalue $\lambda\neq 0$. Then \Cref{prop:smoothSol2} implies (with $m=\kappa+1$) the existence of an history function $\history$ such that the solution of the IVP \eqref{eq:IVP} satisfies
\begin{align*}
	\derv{\mos{\fstate}{2}}{\kappa+1}(0^+)-\derv{\mos{\fstate}{1}}{\kappa+1}(\delay^-) &= \Baa\left(\derv{\fhistory}{\kappa+1}(0)-\derv{\mos{\fstate}{1}}{\kappa+1}(0^+)\right)\\
		&=\lambda\tilde{\fstate}\neq 0.
\end{align*}
Thus, in general we cannot expect the solution of a DDAE of \neutral{} to get any smoother, which again justifies the terminology. For DDAEs of \advanced{}, \Cref{ex:advanced} might suggest that the solution becomes less and less smooth until it becomes discontinuous. This is however not necessarily the case as the following example demonstrates.

\begin{example}
	\label{ex:weekleDeSmoothing}
	Suppose that the DDAE \eqref{eq:DDAE} satisfies \Cref{ass:regular,ass:smoothness} and additionally satisfies $\N\Baa=0$, $\N\Ba\neq0$, and $\N^2\Ba=0$, i.\,e., the DDAE is of \advanced{} according to \Cref{thm:classificationDDAE}. Suppose that the history function $\history$ satisfies \eqref{eq:smoothnessCondition}. Then
	\begin{align*}
		\mos{\fstate}{2}(0) - \mos{\fstate}{1}(0) &= \sum_{k=0}^{\ind(\E,\A)-1}\N^k\Ba\left(\derv{\history}{k}(0)-\derv{\mos{\state}{1}}{k}(0^+)\right)\\
		&= \sum_{k=0}^1 \N^k\Ba\left(\derv{\history}{k}(0)-\derv{\mos{\state}{1}}{k}(0^+)\right) = 0.
	\end{align*}
	However, we have $\dmos{\sstate}{2}(0^+) - \dmos{\fstate}{1}{\delay^-} = 0$ by the definition of the slow system \eqref{eq:slowSystem} and by induction we infer
	\begin{displaymath}
		\mos{\fstate}{i+1}(0^+) - \mos{\fstate}{i}(\delay^-) = \N\Bad\left(\dmos{\sstate}{i-1}(\delay^-)-\dmos{\sstate}{i}(0^+)\right) = 0.
	\end{displaymath}
	Thus the initial condition $\mos{\state}{i}(0) = \mos{\state}{i-1}(\delay)$ is consistent for \eqref{eq:sequenceDAEs} and hence the solution exists for all $\tf>0$.
\end{example}

For a general analysis let us assume that the DDAE \eqref{eq:DDAE} satisfies \Cref{ass:regular,ass:smoothness} and is of \advanced{} and the history function $\history$ satisfies the splicing condition \eqref{eq:splicingCondition}. From \eqref{eq:slowSystem} we infer inductively
\begin{displaymath}
	\derv{\mos{\sstate}{2}}{k}(0^+) = \J\derv{\mos{\sstate}{2}}{k-1}(0^+) + \Bd\derv{\mos{\state}{1}}{k-1}(0^+) + \derv{\smosf{2}}{k}(0) = \derv{\mos{\sstate}{1}}{k}(\delay^-)
\end{displaymath}
for $k=1,\ldots,\kappa+1$. For the fast subsystem, the splicing condition \eqref{eq:splicingCondition} implies
\begin{align*}
	\mos{\fstate}{2}(0^+)-\mos{\fstate}{1}(\delay^-) = \sum_{k=\kappa+1}^{\ind(\E,\A)-1}\N^k\Ba \left(\derv{\history}{k}(0) - \derv{\mos{\state}{1}}{k}(0^+)\right)
\end{align*}
and hence a sufficient condition for the initial condition $\mos{\fstate}{2}(0)=\mos{\fstate}{1}(\delay)$ to be consistent is to assume $\N^k\Ba=0$ for $k\geq \kappa+1$. Note that this is immediately satisfied for $\ind(\E,\A)\leq \kappa+1$. To analyze the next interval we compute 
\begin{align*}
	\mos{\fstate}{3}(0^+) - \mos{\fstate}{2}(\delay^-) &= \sum_{k=1}^{\kappa}\N^k\Ba\WCFR\begin{bmatrix}
		\derv{\mos{\sstate}{1}}{k}(\delay^-) - \derv{\mos{\sstate}{2}}{k}(0^+)\\
		\derv{\mos{\fstate}{1}}{k}(\delay^-)-\derv{\mos{\fstate}{2}}{k}(0^+)
	\end{bmatrix}\\
	&= \sum_{k=1}^{\kappa} \N^k\Baa\left(\derv{\mos{\fstate}{1}}{k}(\delay^-)-\derv{\mos{\fstate}{2}}{k}(0^+)\right).
\end{align*}
Thus, the assumption $\N\Baa=0$ implies $\mos{\fstate}{3}(0^+) - \mos{\fstate}{2}(\delay^-) = 0$. Unfortunately, we have
\begin{align*}
	\derv{\mos{\sstate}{3}}{2}(0^+) - \derv{\mos{\sstate}{2}}{2}(\delay^-) &= \Bda\left(\dmos{\fstate}{2}(0^+) - \dmos{\fstate}{1}(\delay^-)\right),
\end{align*}
and thus cannot show that the initial condition $\mos{\fstate}{4}(0) = \mos{\fstate}{3}(\delay)$ is consistent without posing further assumptions on the matrices $\E,\A$, and $\B$. Since this becomes quite technical, we summarize our findings only for the case $\ind(\E,\A)\leq 3$.

\begin{theorem}
	\label{thm:index3} Suppose the IVP \eqref{eq:IVP} satisfies \Cref{ass:regular,ass:smoothness} and $\ind(\E,\A)\leq 3$. Moreover, assume $\N\Baa = 0$ and $\N^2\Bad\Bda = 0$. Then for every admissible history function $\history$ that satisfies \eqref{eq:smoothnessCondition} and 
	\begin{align}
		\label{eq:secondSplicingCondition}
		\ddot{\history}(0) &= \Adiff\dot{\history}(0) + \sum_{k=0}^{\ind\E,\A}\left(\Bk{k}\derv{\history}{k+1}(-\delay) + \C{k}\derv{\f}{k+1}(0)\right)
	\end{align}
	the IVP \eqref{eq:IVP} has a unique solution.
\end{theorem}

\begin{proof}
	The assumptions on $\history$ imply
	\begin{displaymath}
		\derv{\history}{k}(0) = \derv{\state}{k}(0^+)\qquad\text{for}\ k=0,1,2.
	\end{displaymath}
	Since $\ind(\E,\A)\leq 3$, we have $\N^3 = 0$. Together with $\N\Baa=0$ the previous discussion guarantees that a solution exists on the interval $[-\delay,3\delay]$. Using $\N\Baa=0$, we observe
	\begin{align*}
		\mos{\fstate}{i+1}(0^+) - \mos{\fstate}{i}(\delay^-) &= \sum_{k=0}^2 \N^k\Bad\left(\derv{\mos{\sstate}{i-1}}{k}(\delay^-) - \derv{\mos{\sstate}{i}}{k}(0^+)\right)\\
		&= \N^2\Bad\Bda\left(\dmos{\fstate}{i-2}(\delay^-) - \dmos{\fstate}{i-1}(0^+)\right) = 0
	\end{align*}
	and thus the initial condition $\mos{\state}{i+1}(0) = \mos{\state}{i}(\tau)$ is consistent for all $i\in\mathbb{N}$. The result follows from \Cref{cor:consistencyCondition}.
\end{proof}

\begin{remark}
	The proof of \Cref{thm:index3} shows that the result can be further improved by requiring different splicing conditions for the history function $\shistory$ for the slow state $\sstate$ and for the history function $\fhistory$ of the fast state $\fstate$. 
\end{remark}

\section{Comparison to the existing classification}

In \cite{HaM15} the authors replace the delayed argument in the DDAE \eqref{eq:DDAE} with a function parameter $\lambda:\timeInt\to\field^{\stateDim}$ and obtain the initial value problem
\begin{equation}
	\label{eq:functionParameterDAE}
	\begin{aligned}
		\E\dot{\state}(t) &= \A\state(t) + \B\lambda(t) + \f(t),\\
		\state(t) &= \history(0),
	\end{aligned}
\end{equation}
on the time interval $\timeInt$. They call the function parameter $\lambda$ \emph{consistent} if there exists a consistent initial condition $\history(0)$ for the IVP \eqref{eq:functionParameterDAE}. Based on the function parameter $\lambda$ the following classification for DDAEs \cite{HaM15} is introduced.

\begin{defn}
	\label{def:HaClassification}
	The DDAE \eqref{eq:DDAE} is called \emph{retarded}, \emph{neutral}, or \emph{advanced}, if the minimum smoothness requirement for a consistent function parameter $\lambda$ is that $\lambda\in\mathcal{C}^{1}(\timeInt,\field^{\stateDim})$, $\lambda\in\mathcal{C}^1(\timeInt,\field^{\stateDim})$, or $\lambda\in\mathcal{C}^{k}(\timeInt,\field^{\stateDim})$ for some $k\geq 2$.
\end{defn}

To compare the classification based on propagation of primary discontinuities (cf. \Cref{def:classification}) with the classification of \cite{HaM15}, we need to understand \Cref{def:HaClassification} in terms of the quasi-Weierstra{\ss} form.

\begin{prop}
	Suppose that the DDAE \eqref{eq:DDAE} satisfies \Cref{ass:regular,ass:smoothness}. Then the DDAE \eqref{eq:DDAE} is  
	\begin{itemize}
		\item retarded if and only if $\Ba=0$, 
		\item neutral if and only if $\Ba\neq0$ and $\N\Ba=0$, and
		\item advanced otherwise,
	\end{itemize}
	where $\Ba$ and $\N$ are the matrices from the quasi-Weierstra{\ss} form (\Cref{thm:quasiWf}) and \eqref{eq:WCF}.
\end{prop}

\begin{proof}
	The smoothness requirements for $\lambda$ can be directly seen from the underlying DDE \eqref{eq:underlyingDDE}. Note that we have
	\begin{align*}
		\Bk{0} &= \WCFR\begin{bmatrix}
			I_{\ndif} &0\\
			0&0
		\end{bmatrix}\WCFL\B = \WCFR\begin{bmatrix}
			\Bd\\0
		\end{bmatrix}\qquad\text{and}\\
		\Bk{k} &= -\WCFR\begin{bmatrix}
			0&0\\
			0&\N^{k-1}
		\end{bmatrix}\WCFL\B = -\WCFR\begin{bmatrix}
			0\\
			\N^{k-1}\Ba
		\end{bmatrix}
	\end{align*}
	for $k=1,\ldots,\ind(\E,\A)$. Hence \eqref{eq:DDAE} is retarded if and only if $\N^{k-1}\Ba=0$ for all $k=1,\ldots,\ind(\E,\A)$, which is equivalent to $\Ba=0$. The DDAE is neutral, if $\N^{k-1}\Ba=0$ for all $k=2,\ldots,\ind(\E,\A)$, which is equivalent to $\N\Ba=0$ and otherwise advanced.
\end{proof}

With the characterization we see immediately that the classification by \cite{HaM15} provides in the following sense an upper bound for the new definition.

\begin{corollary}
	\label{cor:classification}
	Suppose that the DDAE \eqref{eq:DDAE} satisfies \Cref{ass:regular,ass:regular}.
	\begin{itemize}
		\item If \eqref{eq:DDAE} is not advanced, then the DDAE \eqref{eq:DDAE} is not of \advanced{}. 
		\item If the DDAE \eqref{eq:DDAE} is advanced, then it is of \advanced{}.
	\end{itemize}
\end{corollary}

Since the classification introduced in this paper is based on the worst-case scenario, the numerical method described in \cite{HaM15}, which is formulated for DDAEs that are not advanced, is safe to use.

\begin{remark}
	The numerical method introduced in \cite{HaM15} is tailored to DDAEs that are not advanced and cannot be used for advanced DDAEs. However, if it is known that the history function satisfies splicing conditions of the form \eqref{eq:smoothnessCondition} and \eqref{eq:secondSplicingCondition}, then also advanced DDAEs may be solved (cf. \Cref{thm:index3}). Thus, there is a need for numerical integration schemes that can handle such situations. This is subject to further research.
\end{remark}

\section{Summary}
In this paper we have studied the propagation of primary discontinuities in initial value problems for delay differential-algebraic equations. Based on the the different possible propagation types we have introduced a new classification for DDAEs and developed a complete characterization in terms of the coefficient matrices. Moreover, the analysis shows that hidden delays may be possible in DDAEs and we have introduced a systematic way to reformulate the DDAE in terms of these delays. As a consequence, we showed that the stability analysis for such DDAEs can be performed by computing the spectral abscissa. In addition, we have studied the impact of splicing conditions on the classification and derived sufficient conditions for DDAEs of index less or equal three to have a unique solution.

\bibliographystyle{plain}
\bibliography{ClassificationDDAE}    

\end{document}